\documentclass[11pt]{article}
\usepackage[T1]{fontenc}
\usepackage{lmodern}
\usepackage{amsmath,amssymb,amsthm,mathtools}
\usepackage{booktabs}
\usepackage{enumitem}
\usepackage{microtype}
\usepackage[margin=1.08in]{geometry}
\usepackage[hidelinks]{hyperref}
\usepackage{xurl}

\newtheorem{theorem}{Theorem}[section]
\newtheorem{proposition}[theorem]{Proposition}
\newtheorem{lemma}[theorem]{Lemma}

\theoremstyle{definition}

\title{Gaussian Vertex--Face Balance in Random Convex Polyhedra\\with Fixed Edge Count}
\author{Lachlan Bridges\\[3pt]
\small Independent researcher, Adelaide, South Australia, Australia\\
\small \href{mailto:lachlanjbridges@gmail.com}{lachlanjbridges@gmail.com}}
\date{16 September 2026}

\begin{document}
\maketitle

\begin{abstract}
Choose uniformly among the combinatorial types of convex three-dimensional
polyhedra with a fixed admissible number $e$ of edges, and let $V_e$ be the
number of vertices.  This resolves a fixed-edge limit-distribution question
posed by R\"udinger: if $\beta_e=V_e/(e+2)$, then
$\sqrt e(\beta_e-1/2)\Rightarrow N(0,1/32)$, equivalently
$\operatorname{Var}(V_e)\sim e/32$.  Using the classical rooted enumeration and asymmetry results of Bender and
Wormald, we derive a relative lattice local limit
theorem on every $o(e^{3/4})$ window, a quartic correction from the rate
function on every $o(e^{5/6})$ window, precise moderate-tail constants, a
quadratic moderate-deviation principle for every $a_e\to\infty$ with
$a_e=o(\sqrt e)$, and a full speed-$e$ large-deviation principle with an
explicit good rate function.  All fixed standardised moments converge.  When
$e=3m$, the two extremal vertex counts have equal probability asymptotic to
$\frac{6561}{32\sqrt2}(4/27)^m$.  Rooted and unrooted fixed-edge laws are also
uniformly exponentially close, with relative discrepancy $O(\rho^e)$.
\end{abstract}


\bigskip
\section{Introduction and main results}\label{sec:intro}

Fix the number of edges of a convex three-dimensional polyhedron and choose
uniformly among its combinatorial types.  How balanced are the numbers of
vertices and faces?  R\"udinger posed this fixed-edge distribution problem in
2009 and 2010, using the form parameter
\[
  \beta=\frac{V}{e+2}
\]
and asking in particular whether a non-trivial limiting law appears after
rescaling \cite{Ruedinger2009,Ruedinger2010}.  Numerical data suggested
concentration around $1/2$, but did not determine the fluctuation scale.

We answer that question explicitly.  If $V_e$ denotes the number of vertices
of a uniformly chosen combinatorial type with $e$ edges, then
\begin{equation}
  \sqrt e\left(\frac{V_e}{e+2}-\frac12\right)
  \Rightarrow N(0,1/32).
  \label{eq:ruedinger-answer}
\end{equation}
Equivalently, the centred vertex count has variance asymptotic to $e/32$.
Beyond the central limit, the same fixed-edge analysis yields a relative
lattice local limit theorem, a quartic correction from the rate-function
expansion, precise moderate-tail constants, a moderate-deviation principle on
every scale $a_e\to\infty$ with $a_e=o(\sqrt e)$, a full speed-$e$
large-deviation principle, convergence of all fixed moments, and sharp
probabilities at the two extremal vertex counts when $3\mid e$.

The enumerative inputs are classical.  Bender and Wormald obtained a uniform
two-parameter asymptotic for rooted $3$-connected planar maps
\cite{BenderWormald1988}, while their earlier asymmetry theorem gives uniform
control of the passage from rooted to unrooted maps
\cite{BenderWormald1985}.  We first make that dependency explicit and then carry
out the fixed-edge Laplace analysis.  The resulting probability asymptotics range from central
Gaussian fluctuations through large deviations to extremal configurations.

The results below require estimates that remain uniform as the fluctuation
scale grows from
$O(\sqrt e)$ through moderate deviations, while also controlling the shifted
integer/half-integer lattice and the fixed-density outer region.  At the
support boundary the bulk prefactor $G$ vanishes, so the extremal
probabilities cannot be obtained by extrapolating the interior asymptotic and
require a separate boundary calculation.  Likewise, allowing the MDP scale
$a_e$ to diverge arbitrarily slowly rules out a single-atom lower bound and
forces a shift-uniform mesh-interval argument.  These uniformity and boundary
issues are what connect the central Gaussian regime to the full fixed-edge
asymptotic hierarchy.

Limit laws for other random planar graph and map parameters form a broad
literature; representative examples include the planar-graph limit laws of
Gim\'enez and Noy \cite{GimenezNoy2009} and the degree-count limit laws for
planar maps of Collet, Drmota and Klausner \cite{ColletDrmotaKlausner2019}.
Those models and parameters differ from the present uniform fixed-edge,
unlabelled polyhedral model, but they provide useful probabilistic context.

\subsection{The model}

Let
\[
  \mathcal E:=\{6\}\cup\{e\in\mathbb Z:e\ge 8\}.
\]
For $e\in\mathcal E$, let $\mathcal P_e$ be the finite set of combinatorial
types of convex $3$-polytopes with $e$ edges, and choose $\Pi_e$ uniformly from
$\mathcal P_e$.  Write $V_e$ and $F_e$ for the numbers of vertices and faces of
$\Pi_e$.  The set $\mathcal E$ is exactly the set of possible edge counts:
a pyramid over an $n$-gon has $2n$ edges, giving every even value $e\ge6$,
and stacking a new vertex beyond one triangular side facet adds three edges,
giving every odd value $e\ge9$ from the corresponding $(e-3)$-edge pyramid.
On the other hand, no convex $3$-polytope has seven edges: if $e=7$, then
minimum vertex and face degree give $V,F\le4$, contradicting Euler's identity
$V+F=e+2=9$.  Thus the fixed-edge model is defined precisely for
$e\in\mathcal E$.

Set
\[
  i=V_e-1,\qquad j=F_e-1,\qquad i+j=e,
\]
and to centre by
\[
  X_e:=V_e-\frac{e+2}{2}.
\]
The variable $X_e$ takes values in the translated unit lattice
\[
  \Lambda_e=\begin{cases}
  \mathbb Z,& e\text{ even},\\
  \mathbb Z+\tfrac12,& e\text{ odd},
  \end{cases}
\]
but for each fixed $e$ its feasible support is a finite subset of
$\Lambda_e$.  Below, ``attainable'' means belonging to this finite support.
For density-scale statements we use
\[
  A_e:=\frac{V_e-1}{e}=\frac{i}{e}.
\]

Define, for $\alpha\in[1/3,2/3]$,
\begin{align}
 I(\alpha)
  &=(3\alpha-1)\log(3\alpha-1)
    +(2-3\alpha)\log(2-3\alpha) \notag\\
  &\qquad -\alpha\log\alpha-(1-\alpha)\log(1-\alpha),
  \label{eq:rate-I}
\end{align}
where $0\log 0:=0$, and put $I(\alpha)=+\infty$ outside
$[1/3,2/3]$.  We also set
\begin{equation}
  G(\alpha):=
  \frac{\bigl((3\alpha-1)(2-3\alpha)\bigr)^{5/2}}
       {\alpha^4(1-\alpha)^4},
  \qquad
  \Phi(\alpha):=\log 4-I(\alpha).
  \label{eq:G-Phi}
\end{equation}

\subsection{Main theorem}

\begin{center}
\fbox{\begin{minipage}{0.88\linewidth}
\textbf{Central consequence.}
If $\beta_e=V_e/(e+2)$, then
\[
  \sqrt e\,(\beta_e-\tfrac12)\Rightarrow N(0,1/32),
  \qquad
  \operatorname{Var}(V_e)\sim e/32.
\]
Thus R\"udinger's fixed-edge form parameter has Gaussian fluctuations on
the $e^{-1/2}$ scale.  The theorem below gives the sharper local, deviation,
moment and endpoint structure surrounding this limit.
\end{minipage}}
\end{center}
\medskip

\begin{theorem}[Fixed-edge vertex--face asymptotics]\label{thm:main}
As $e\to\infty$ through $\mathcal E$, the following statements hold.
\begin{enumerate}[label=\textup{(\roman*)},leftmargin=2.4em]
\item \emph{Exact duality.}  For every feasible integer $v$,
\[
  \mathbb P(V_e=v)=\mathbb P(V_e=e+2-v).
\]
Consequently $\mathbb E V_e=(e+2)/2$ and every odd centred moment is exactly
zero.

\item \emph{Sharp bulk law.}  Uniformly for $\alpha=i/e$ in every compact
subset of $(1/3,2/3)$,
\begin{equation}
 \mathbb P(V_e=i+1)=
 \frac{G(\alpha)}{2\sqrt{\pi e}}\,
 e^{-eI(\alpha)}(1+o(1)).
 \label{eq:bulk-prob}
\end{equation}
The function $I$ is symmetric about $1/2$, strictly convex on the interior,
and has its unique zero at $1/2$, with $I''(1/2)=32$.

\item \emph{Local Gaussian and quartic laws.}  For every
$r_e=o(e^{3/4})$,
\begin{equation}
 \sup_{\substack{x\in\Lambda_e\\ |x|\le r_e}}
 \left|
 \frac{\mathbb P(X_e=x)}{\frac{4}{\sqrt{\pi e}}
       \exp(-16x^2/e)}-1
 \right|\longrightarrow 0.
 \label{eq:relative-gauss}
\end{equation}
For every $s_e=o(e^{5/6})$, uniformly for
$x\in\Lambda_e$ with $|x|\le s_e$,
\begin{equation}
 \mathbb P(X_e=x)=\frac{4}{\sqrt{\pi e}}
 \exp\left\{-\frac{16x^2}{e}-\frac{320x^4}{3e^3}+o(1)\right\}.
 \label{eq:quartic}
\end{equation}
Moreover, with $\varphi(z)=(2\pi)^{-1/2}e^{-z^2/2}$ and the convention
$\mathbb P(X_e=x)=0$ at non-attainable lattice points,
\begin{equation}
 \sup_{x\in\Lambda_e}
 \left|
 \sqrt{\frac e{32}}\,\mathbb P(X_e=x)
 -\varphi\!\left(\frac{x}{\sqrt{e/32}}\right)
 \right|\longrightarrow 0.
 \label{eq:global-llt}
\end{equation}

\item \emph{Precise moderate tails.}  If
$t_e/\sqrt e\to\infty$ and $t_e=o(e^{3/4})$, and $t_e$ is rounded upward to
the first point of $\Lambda_e$ when necessary, then
\begin{align}
 \mathbb P(X_e\ge t_e)
 &\sim \frac{\sqrt e}{8\sqrt\pi\,t_e}
       \exp(-16t_e^2/e), \label{eq:one-tail}\\
 \mathbb P(|X_e|\ge t_e)
 &\sim \frac{\sqrt e}{4\sqrt\pi\,t_e}
       \exp(-16t_e^2/e). \label{eq:two-tail}
\end{align}

\item \emph{Moderate deviations.}  For every sequence
$a_e\to\infty$ with $a_e=o(\sqrt e)$, including arbitrarily slow divergence,
\[
  Y_e:=\frac{X_e}{a_e\sqrt{e/32}}
\]
satisfies a large-deviation principle with speed $a_e^2$ and good rate
$J(y)=y^2/2$.

\item \emph{Large deviations, CLT and moments.}  The variables $A_e$ satisfy
a large-deviation principle with speed $e$ and good rate $I$.  Also
\[
  \frac{X_e}{\sqrt{e/32}}\Rightarrow N(0,1),
  \qquad
  \operatorname{Var}(V_e)\sim\frac e{32}.
\]
Every fixed even standardised moment converges to the corresponding Gaussian
moment $(2r-1)!!$, while every fixed odd centred moment is identically zero.
In particular, if $\beta_e:=V_e/(e+2)$, then
\[
  \sqrt e\,(\beta_e-\tfrac12)\Rightarrow N(0,1/32).
\]

\item \emph{Extremal vertex counts.}  If $e=3m$, $m\ge2$, then
\[
  m+2\le V_{3m}\le 2m,
\]
both endpoints occur, and polar duality gives exact equality of their
probabilities.  As $m\to\infty$,
\begin{equation}
 \mathbb P(V_{3m}=m+2)
 =\mathbb P(V_{3m}=2m)
 \sim \frac{6561}{32\sqrt2}\left(\frac4{27}\right)^m.
 \label{eq:endpoint-prob}
\end{equation}
At the lower endpoint every face is triangular; at the upper endpoint every
vertex has degree three.

\item \emph{Rooted--unrooted transfer.}  Let $R_e$ be the distribution of
$V$ under the corresponding uniformly rooted model, and let $P_e$ denote the
unrooted law above.  There exists $\rho\in(0,1)$ such that, uniformly on the
complete support,
\[
  R_e(V=v)=P_e(V=v)\bigl(1+O(\rho^e)\bigr),
  \qquad
  \|R_e-P_e\|_{\mathrm{TV}}=O(\rho^e).
\]
The same multiplicative estimate holds for every positive-probability event and
for every non-negative observable with positive expectation.  Consequently all local laws, precise tails,
endpoint asymptotics, fixed even standardised moment asymptotics, and all
LDP/MDP statements above transfer to the rooted model.  The signed-observable
bound in Section~\ref{sec:transfer} also transfers convergence of the odd
standardised moments to zero, but does not assert exact finite-$e$ vanishing
for the rooted law.  For a positive event $B_e$,
\[
  |\log R_e(B_e)-\log P_e(B_e)|=O(\rho^e).
\]
\end{enumerate}
\end{theorem}

The scales in Theorem~\ref{thm:main} are summarised in
Table~\ref{tab:scales}.

\begin{table}[ht]
\centering
\caption{Fixed-edge asymptotic regimes for the centred vertex count $X_e$.}
\label{tab:scales}
\begin{tabular}{@{}lll@{}}
\toprule
Regime & Scale / window & Asymptotic description \\
\midrule
Central & $x=O(\sqrt e)$ & Gaussian, variance $e/32$ \\
Relative local & $|x|=o(e^{3/4})$ & Gaussian point probabilities \\
Quartic local & $|x|=o(e^{5/6})$ & Gaussian plus $-320x^4/(3e^3)$ \\
Moderate tail & $\sqrt e\ll t=o(e^{3/4})$ & Mills-type exact prefactor \\
MDP & $X_e\asymp a_e\sqrt e$ & speed $a_e^2$, rate $y^2/2$ \\
LDP & $X_e\asymp e$ & speed $e$, rate $I$ \\
Endpoint & $e=3m$ & constant $\times(4/27)^m$ \\
\bottomrule
\end{tabular}
\end{table}

\section{Model, duality and classical enumeration inputs}\label{sec:model}

By Steinitz's theorem, the graphs of convex $3$-polytopes are precisely the
$3$-connected planar graphs; see, for example, \cite[Ch.~4]{Ziegler1995}.
Euler's formula gives $i+j=e$ in our shifted notation.  Since every vertex of a
$3$-polytope has degree at least three and the same is true in the dual,
\begin{equation}
  3V\le2e,\qquad 3F\le2e.
  \label{eq:support-ineq}
\end{equation}
Consequently every feasible density $\alpha=i/e$ lies in
$[1/3+O(e^{-1}),2/3+O(e^{-1})]$.

\subsection{Exact duality}

\begin{proposition}\label{prop:duality}
For each $e\in\mathcal E$, polar duality is a bijection on $\mathcal P_e$ that
maps $V$ to $e+2-V$.  Hence the law of $X_e$ is exactly symmetric.
\end{proposition}

\begin{proof}
Choose a geometric realisation of a combinatorial type with the origin in its
interior and take the polar.  Vertices and faces are interchanged, edges are
preserved, and polarity is an involution on combinatorial types.  Uniform
measure is therefore invariant under this bijection.  Since $F=e+2-V$, the
asserted distributional identity follows.  Pairing dual types (or using
invariance directly) gives the exact mean; symmetry of $X_e$ gives all odd
centred moments.
\end{proof}

\subsection{The two classical inputs}

Let $p_{i,j}$ denote the rooted count with $i+1$ vertices and $j+1$ faces.
Bender and Wormald \cite{BenderWormald1988} proved, uniformly as
$\max(i,j)\to\infty$,
\begin{equation}
 p_{i,j}\sim
 \frac{1}{3^5ij}
 \binom{2i}{j+3}\binom{2j}{i+3}.
 \label{eq:bw-rooted}
\end{equation}
The exact indexing and the two $+3$ shifts will matter below.  The stated
uniformity has no compact-ratio restriction, so it applies to the moving
fixed-edge sequences used here, including sequences approaching the feasible
boundary.

The second input is the uniform asymmetry theorem of Bender and Wormald
\cite{BenderWormald1985}.  Let $q_{v,f}$ be the rooted count with $v$ vertices
and $f$ faces, and let $s_{v,f}$ be the number of symmetric unrooted maps.  There
exists $c\in(0,1)$ such that
\begin{equation}
  s_{v,f}=o(c^v q_{v,f})
  \quad\text{uniformly for}\quad
  \frac12<\frac fv<2.
  \label{eq:bw-asymmetry}
\end{equation}
Their Corollary~4.2 gives the corresponding leading unrooting factor
$4(v+f-2)$.  The fixed-edge support lies strictly inside the ratio range in
\eqref{eq:bw-asymmetry}: from \eqref{eq:support-ineq} and
$v+f=e+2$ one has $v,f\ge e/3+2$ and $v,f\le2e/3$.  In particular,
\[
  2f-v=2e+4-3v\ge4>0,
  \qquad
  2v-f=3v-e-2\ge4>0,
\]
so explicitly
\[
  \frac12<\frac fv<2
\]
at every feasible fixed-edge state.

Immediately after their rooted theorem, Bender and Wormald
\cite{BenderWormald1988} also print the associated unrooted asymptotic and
explicitly note that it follows from the earlier asymmetry/unrooting result.
We therefore derive the fixed-edge unrooted asymptotic from these two inputs
in the next section.  Earlier enumerative work includes
\cite{MullinSchellenberg1968,Tutte1980,RichmondWormald1982,BenderRichmond1984}.

\section{Uniform rooted--unrooted comparison on a fixed-edge slice}\label{sec:transfer}

Let $u_{v,f}$ be the number of unrooted combinatorial types with $v$ vertices
and $f$ faces, and put $e=v+f-2$.

\begin{lemma}[Exact rooting identity]\label{lem:rooting}
If $M$ is an unrooted $e$-edge map, the number of distinct rootings of $M$ is
\[
 \frac{4e}{|\operatorname{Aut}(M)|}.
\]
Consequently
\begin{equation}
 0\le u_{v,f}-\frac{q_{v,f}}{4e}\le s_{v,f}.
 \label{eq:rooting-sandwich}
\end{equation}
\end{lemma}

\begin{proof}
There are $4e$ oriented edge-side flags before quotienting by automorphisms,
and the stabiliser of a rooted flag is trivial.  Orbit--stabiliser therefore
gives $4e/|\operatorname{Aut}(M)|$ distinct rootings.  An asymmetric type
contributes exactly one to $q_{v,f}/(4e)$, while a symmetric type contributes
$1/|\operatorname{Aut}(M)|\le1$.  Summing the deficit over symmetric types
gives \eqref{eq:rooting-sandwich}.
\end{proof}

\begin{theorem}[Exponentially accurate fixed-edge transfer]\label{thm:transfer}
There exists $\rho\in(0,1)$ such that, uniformly over the complete feasible
fixed-edge support,
\begin{equation}
 q_{v,f}=4e\,u_{v,f}\bigl(1+O(\rho^e)\bigr).
 \label{eq:q-u}
\end{equation}
The rooted and unrooted count systems have the same support.  If $R_e$ and
$P_e$ denote their normalised vertex laws, then uniformly on the common
support
\begin{equation}
  \frac{R_e(V=v)}{P_e(V=v)}=1+O(\rho^e),
  \qquad
  \|R_e-P_e\|_{\mathrm{TV}}=O(\rho^e).
  \label{eq:law-ratio}
\end{equation}
The same multiplicative estimate holds for every positive-probability event
and every non-negative observable with positive expectation.
\end{theorem}

\begin{proof}
The exact rooting identity shows first that $q_{v,f}=0$ if and only if
$u_{v,f}=0$.  At a positive-count point define
\[
 d_{v,f}:=u_{v,f}-\frac{q_{v,f}}{4e}\ge0.
\]
By \eqref{eq:rooting-sandwich} and \eqref{eq:bw-asymmetry}, uniformly on the
fixed-edge support,
\[
  0\le \frac{4e\,d_{v,f}}{q_{v,f}}
  \le \frac{4e\,s_{v,f}}{q_{v,f}}
  =o(e c^v).
\]
Since $v\ge e/3+2$, choose any $\rho\in(c^{1/3},1)$.  The polynomial factor
$e$ is then absorbed into $O(\rho^e)$, proving \eqref{eq:q-u} uniformly.

Write $q_v=4e\,u_v(1+\varepsilon_{e,v})$ with
$\eta_e:=\sup_v|\varepsilon_{e,v}|=O(\rho^e)$.  If
$\bar\varepsilon_e=\sum_vP_e(v)\varepsilon_{e,v}$, then
$|\bar\varepsilon_e|\le\eta_e$ and
\[
  R_e(v)=P_e(v)\frac{1+\varepsilon_{e,v}}
                         {1+\bar\varepsilon_e}.
\]
This gives the pointwise ratio and, after summation, the total-variation
bound.  Averaging the same ratio over an event or against a non-negative
observable gives the final assertion.
\end{proof}

For any positive-probability event $B_e$, Theorem~\ref{thm:transfer} also gives
\begin{equation}
  R_e(B_e)=P_e(B_e)(1+O(\rho^e)),
  \qquad
  |\log R_e(B_e)-\log P_e(B_e)|=O(\rho^e).
  \label{eq:event-transfer}
\end{equation}
Thus every LDP or MDP at a diverging speed transfers unchanged between the two
models, as do the local laws, precise positive tails, endpoint probabilities,
and fixed non-negative moments.  For signed observables the corresponding
absolute error is $O(\rho^e)$ times the unrooted absolute moment.  In
particular, because the unrooted odd centred moments vanish exactly by duality,
the rooted odd standardised moments converge to zero (indeed with an
$O(\rho^e)$ relative-to-absolute-moment error bound), although exact
finite-$e$ vanishing is not claimed for the rooted model.

Combining \eqref{eq:bw-rooted} with \eqref{eq:q-u}, and recalling
$q_{i+1,j+1}=p_{i,j}$ and $e=i+j$, define the shifted unrooted count
\[
  \widetilde u_{i,j}:=u_{i+1,j+1}.
\]
Then the fixed-edge unrooted asymptotic is
\begin{equation}
 \widetilde u_{i,j}\sim
 \frac{1}{4\,3^5ij(i+j)}
 \binom{2i}{j+3}\binom{2j}{i+3},
 \label{eq:bw-unrooted}
\end{equation}
uniformly over the complete feasible fixed-edge support as $e\to\infty$.
This is the same unrooted formula printed in \cite{BenderWormald1988}; the
point of the derivation here is to make its reliance on the earlier asymmetry
result explicit.  The exponential transfer error is much smaller than the
source theorem's unquantified uniform $o(1)$ and hence does not alter the
leading asymptotic.

\begin{lemma}[Interior support]\label{lem:interior-support}
Let $K\subset(1/3,2/3)$ be compact.  For all sufficiently large admissible
$e$, every integer pair $i+j=e$ with $i/e\in K$ is feasible; equivalently,
$\widetilde u_{i,j}>0$.  In particular, if $r_e=o(e)$, then every
$x\in\Lambda_e$ with $|x|\le r_e$ is attainable for all sufficiently large
$e$.
\end{lemma}

\begin{proof}
Because $K$ is compactly contained in $(1/3,2/3)$, for all sufficiently large
$e$ every such pair satisfies
\[
  j+3\le 2i,\qquad i+3\le 2j,
\]
and $i,j>0$.  Hence both binomial coefficients on the right-hand side of
\eqref{eq:bw-rooted} are positive.  The Bender--Wormald asymptotic is uniform
as $\max(i,j)\to\infty$, so uniformly over these pairs the ratio of
$p_{i,j}$ to that positive right-hand side tends to one.  Thus $p_{i,j}>0$
for all sufficiently large $e$.  The exact support equivalence in
Theorem~\ref{thm:transfer} gives $\widetilde u_{i,j}>0$.

For the final assertion, $i=e/2+x$ gives $i/e=1/2+x/e$.  If $r_e=o(e)$,
these ratios lie in a fixed compact subinterval of $(1/3,2/3)$ for all
sufficiently large $e$, so the first assertion applies.
\end{proof}

\section{Fixed-edge enumeration and normalisation}\label{sec:normalisation}

We now extract a one-dimensional Laplace problem from the fixed-edge
unrooted asymptotic \eqref{eq:bw-unrooted}, whose classical provenance was
made explicit in Section~\ref{sec:transfer}.  Put
\begin{equation}
 h(\alpha):=\frac{G(\alpha)}{4\,3^5\pi}.
 \label{eq:hdef}
\end{equation}

\begin{lemma}[Rate-function geometry]\label{lem:rate}
On $(1/3,2/3)$,
\begin{equation}
 I''(\alpha)=
 \frac{2}{\alpha(1-\alpha)(3\alpha-1)(2-3\alpha)}.
 \label{eq:Isecond}
\end{equation}
Thus $I$ is strictly convex, satisfies $I(1-\alpha)=I(\alpha)$, and has its
unique minimum $I(1/2)=0$ at $1/2$.  Moreover
\begin{align}
 I(\tfrac12+d)
 &=16d^2+\frac{320}{3}d^4+\frac{23296}{15}d^6+O(d^8),
 \label{eq:I-taylor}\\
 G(\tfrac12)&=8,
 \qquad h(\tfrac12)=\frac{2}{3^5\pi}.
 \label{eq:Gcentre}
\end{align}
At the two endpoints,
\[
 I(\tfrac13)=I(\tfrac23)=\log3-\frac23\log2.
\]
\end{lemma}

\begin{proof}
The symmetry is immediate from \eqref{eq:rate-I}.  Differentiating twice and
simplifying gives \eqref{eq:Isecond}, which is positive throughout the open
support.  Symmetry gives $I'(1/2)=0$ and direct substitution gives
$I(1/2)=0$, proving uniqueness of the minimum.  Taylor expansion at $1/2$
gives \eqref{eq:I-taylor}; \eqref{eq:Gcentre} and the endpoint values follow
by substitution.
\end{proof}

Let
\[
  U_e:=\sum_{i+j=e}\widetilde u_{i,j}
\]
be the number of unrooted combinatorial types with $e$ edges; the equality
uses $i+j=e$ under the shifted convention above.

\begin{theorem}[Diagonal counts]\label{thm:counts}
Uniformly for $\alpha=i/e$ in every compact subset of $(1/3,2/3)$,
\begin{equation}
 \widetilde u_{i,e-i}=e^{-4}h(\alpha)e^{e\Phi(\alpha)}(1+o(1)).
 \label{eq:diag-count}
\end{equation}
The source theorem contributes only its stated uniform $o(1)$; no quantitative
rate for that error is assumed.  Furthermore
\begin{equation}
 U_e\sim \frac{4^e}{486\sqrt\pi\,e^{7/2}}.
 \label{eq:Ue}
\end{equation}
For every fixed $\delta\in(0,1/6)$ there are positive constants
$C,c_\delta,\kappa_\delta$ such that, writing $i=e/2+x$ on the attainable
lattice,
\begin{align}
 \widetilde u_{i,e-i}
 &\le C4^e e^{-4}\exp(-c_\delta x^2/e),
 && |x|\le\delta e, \label{eq:count-central-bound}\\
 \widetilde u_{i,e-i}
 &\le C e^6 4^e\exp(-\kappa_\delta e),
 && |x|>\delta e. \label{eq:count-far-bound}
\end{align}
The bounds include both parities and both support endpoints.
\end{theorem}

\begin{proof}
Fix a compact $K\subset(1/3,2/3)$, put $j=e-i$, and let
$\alpha=i/e$.  Uniform Stirling expansion gives
\begin{equation}
 \binom{2i}{j}\binom{2j}{i}
 =\frac{e^{e\Phi(\alpha)}}
 {\pi e\sqrt{(3\alpha-1)(2-3\alpha)}}
 \bigl(1+O_K(e^{-1})\bigr).
 \label{eq:binom-base}
\end{equation}
The fixed shifts can be isolated exactly:
\begin{align}
 \frac{\binom{2i}{j+3}}{\binom{2i}{j}}
 &=\prod_{r=0}^{2}\frac{2i-j-r}{j+r+1}, \label{eq:shift1}\\
 \frac{\binom{2j}{i+3}}{\binom{2j}{i}}
 &=\prod_{r=0}^{2}\frac{2j-i-r}{i+r+1}. \label{eq:shift2}
\end{align}
Uniformly on $K$, these ratios equal
\[
 \left(\frac{3\alpha-1}{1-\alpha}\right)^3(1+O_K(e^{-1}))
 \quad\text{and}\quad
 \left(\frac{2-3\alpha}{\alpha}\right)^3(1+O_K(e^{-1})),
\]
respectively.  Since $ij(i+j)=\alpha(1-\alpha)e^3$, inserting
\eqref{eq:binom-base}--\eqref{eq:shift2} into the unrooted asymptotic
\eqref{eq:bw-unrooted} yields \eqref{eq:diag-count}.  The only unquantified source error is the uniform $o(1)$ inherited from
\eqref{eq:bw-rooted}; the rooted--unrooted correction is exponentially smaller.

We next control the full diagonal.  Strict convexity and
\eqref{eq:Isecond} imply that on
$[1/2-\delta,1/2+\delta]$,
\[
  I(\alpha)\ge c_\delta(\alpha-1/2)^2
\]
for some $c_\delta>0$.  The compact asymptotic therefore gives
\eqref{eq:count-central-bound} after increasing the constant to absorb finitely
many small values of $e$.

For the remaining support we use a bound that is valid up to the endpoints.
Whenever the shifted binomial coefficients are non-zero,
\[
 \binom{2i}{j+3}\le(2e)^3\binom{2i}{j},
 \qquad
 \binom{2j}{i+3}\le(2e)^3\binom{2j}{i}.
\]
With $H(t)=-t\log t-(1-t)\log(1-t)$ and
$\binom nk\le e^{nH(k/n)}$, direct simplification gives
\[
 2\alpha H\!\left(\frac{1-\alpha}{2\alpha}\right)
 +2(1-\alpha)H\!\left(\frac{\alpha}{2(1-\alpha)}\right)
 =\Phi(\alpha).
\]
Thus, using the uniformity of \eqref{eq:bw-unrooted},
\[
  \widetilde u_{i,e-i}\le C e^6 e^{e\Phi(i/e)}.
\]
Since $I$ is continuous on the closed support and strictly positive away from
$1/2$,
\[
 \kappa_\delta=
 \min_{\substack{\alpha\in[1/3,2/3]\\|\alpha-1/2|\ge\delta}}I(\alpha)>0,
\]
which proves \eqref{eq:count-far-bound}.  This argument never divides by the
vanishing interior prefactor at an endpoint.

Finally choose $\eta\in(0,1/4)$ and put $w_e=e^{1/2+\eta}$.  Since
$w_e=o(e)$, Lemma~\ref{lem:interior-support} shows that every
$x\in\Lambda_e$ with $|x|\le w_e$ is attainable for all sufficiently large
$e$.  For these $x$, Lemma~\ref{lem:rate} gives
\[
 eI(\tfrac12+x/e)=\frac{16x^2}{e}+O(x^4/e^3)
 =\frac{16x^2}{e}+o(1)
\]
uniformly.  Hence
\begin{equation}
 \widetilde u_{e/2+x,e/2-x}
 =4^e e^{-4}h(\tfrac12)e^{-16x^2/e}(1+o(1)).
 \label{eq:central-count}
\end{equation}
For either translate of the unit lattice,
\[
 \frac1{\sqrt e}\sum_{x\in\Lambda_e}e^{-16x^2/e}
 \longrightarrow \int_{-\infty}^{\infty}e^{-16t^2}\,dt
 =\frac{\sqrt\pi}{4}.
\]
The tails outside $w_e$ are negligible by
\eqref{eq:count-central-bound}--\eqref{eq:count-far-bound}.  Using
$h(1/2)=2/(3^5\pi)$ in \eqref{eq:central-count} gives
\[
 U_e\sim
 4^e e^{-4}\frac{2}{3^5\pi}\frac{\sqrt{\pi e}}{4}
 =\frac{4^e}{486\sqrt\pi\,e^{7/2}}.
\]
\end{proof}

The total asymptotic \eqref{eq:Ue} is consistent with the classical fixed-edge
enumeration of Richmond and Wormald \cite{RichmondWormald1982}.  Its role here
is as the normalising constant for the distributional results that follow.

\section{Bulk and local asymptotics}\label{sec:local}

\begin{theorem}[Sharp bulk probability]\label{thm:bulk}
Uniformly for $\alpha=i/e$ in every compact subset of $(1/3,2/3)$,
\[
 \mathbb P(V_e=i+1)=
 \frac{G(\alpha)}{2\sqrt{\pi e}}e^{-eI(\alpha)}(1+o(1)).
\]
\end{theorem}

\begin{proof}
Divide \eqref{eq:diag-count} by \eqref{eq:Ue}.  Since
$h(\alpha)=G(\alpha)/(4\,3^5\pi)$ and $486=2\cdot3^5$,
\[
 486\sqrt\pi\,h(\alpha)=\frac{G(\alpha)}{2\sqrt\pi},
\]
which gives the stated prefactor.  All uniformity is inherited from
Theorem~\ref{thm:counts}.
\end{proof}

The bulk formula applies only in the interior.  At the endpoints $G$ vanishes
and a different asymptotic scale applies; Section~\ref{sec:endpoints} treats
those points directly.

\begin{theorem}[Gaussian local law and quartic correction]\label{thm:local}
The local statements \eqref{eq:relative-gauss}--\eqref{eq:global-llt} hold.
\end{theorem}

\begin{proof}
Write $i=e/2+x$ and $d=x/e$.  From Lemma~\ref{lem:rate}, uniformly whenever
$d\to0$,
\begin{equation}
 eI(\tfrac12+d)
 =\frac{16x^2}{e}+\frac{320x^4}{3e^3}
  +\frac{23296x^6}{15e^5}+O(x^8/e^7).
 \label{eq:local-rate-expand}
\end{equation}
Also, by symmetry and smoothness of $G$ near $1/2$,
\[
 G(\tfrac12+d)=8\bigl(1+O(d^2)\bigr).
\]
Both local windows are $o(e)$, so Lemma~\ref{lem:interior-support} shows that
all lattice points under consideration are attainable for all sufficiently
large $e$.  If $|x|\le r_e=o(e^{3/4})$, the quartic and higher terms in
\eqref{eq:local-rate-expand} are $o(1)$, and Theorem~\ref{thm:bulk} yields
\eqref{eq:relative-gauss}.  If $|x|\le s_e=o(e^{5/6})$, the sixth-order and
higher terms are $o(1)$ while the quartic term must be retained, yielding
\eqref{eq:quartic}.

For the global local limit choose a fixed $\eta\in(0,1/4)$.  On
$|x|\le e^{1/2+\eta}$, \eqref{eq:relative-gauss} is uniform and
\[
 \sqrt{\frac e{32}}\frac4{\sqrt{\pi e}}e^{-16x^2/e}
 =\frac1{\sqrt{2\pi}}
   \exp\left\{-\frac12\frac{x^2}{e/32}\right\}.
\]
Outside this window but within a fixed central density interval,
\eqref{eq:count-central-bound} divided by \eqref{eq:Ue} tends uniformly to
zero after multiplication by $\sqrt e$; outside the central density interval,
\eqref{eq:count-far-bound} does the same exponentially.  The Gaussian target
also tends uniformly to zero in these outer regions.  This proves
\eqref{eq:global-llt}.
\end{proof}

The exponents $3/4$ and $5/6$ have a simple origin.  The first neglected term
beyond the quadratic is $x^4/e^3$, so a purely Gaussian relative approximation
requires $x=o(e^{3/4})$; once the quartic term is retained, the next term is
$x^6/e^5$, leading to the larger window $o(e^{5/6})$.

\section{Tails, moderate deviations and global limits}\label{sec:global}

\subsection{Precise moderate tails}

\begin{theorem}[Lattice Mills asymptotics]\label{thm:tails}
If $t_e/\sqrt e\to\infty$ and $t_e=o(e^{3/4})$, then
\eqref{eq:one-tail} and \eqref{eq:two-tail} hold, with upward rounding to
$\Lambda_e$ when necessary.
\end{theorem}

\begin{proof}
Let $\widehat t_e$ be the first point of $\Lambda_e$ not below $t_e$.  Since
$0\le\widehat t_e-t_e<1$,
\[
 \frac{\widehat t_e^2-t_e^2}{e}=o(1),\qquad
 \frac{\widehat t_e}{t_e}\to1,
\]
so threshold rounding is asymptotically harmless.

We make the truncation in the local approximation explicit.  Put
\[
 q_e:=\frac{t_e^2}{e},\qquad
 m_e:=\max\{q_e,\log e\},\qquad
 R_e:=\sqrt{m_e e^{1/2}},\qquad
 r_e:=\sqrt{eR_e}.
\]
Because $t_e=o(e^{3/4})$, one has $m_e=o(e^{1/2})$ and therefore
\[
 q_e+\log e=o(R_e),\qquad t_e=o(r_e),\qquad r_e=o(e^{3/4}).
\]
Since $r_e=o(e)$, Lemma~\ref{lem:interior-support} shows that every
point of $[\widehat t_e,r_e]\cap\Lambda_e$ is attainable for all sufficiently
large $e$.  The relative local theorem is consequently uniform on this entire
lattice interval.

The point bounds from Theorem~\ref{thm:counts}, divided by
\eqref{eq:Ue}, imply for some fixed $C,D,c>0$ that
\[
 \mathbb P(X_e>r_e)\le C e^D e^{-cR_e}.
\]
Since $R_e\gg q_e+\log e$, this is
\[
 o\!\left(\frac{\sqrt e}{t_e}e^{-16q_e}\right),
\]
which is the claimed Mills scale.  The Gaussian lattice tail beyond $r_e$ is
negligible on the same scale by monotone sum--integral comparison.  Hence
\[
 \mathbb P(X_e\ge t_e)
 =\frac4{\sqrt{\pi e}}(1+o(1))
 \sum_{\substack{x\in\Lambda_e\\x\ge\widehat t_e}}
 e^{-16x^2/e}.
\]
Because $t_e/\sqrt e\to\infty$ and the lattice mesh is one,
\[
 \sum_{x\ge\widehat t_e}e^{-16x^2/e}
 \sim \int_{t_e}^{\infty}e^{-16x^2/e}\,dx
 \sim \frac{e}{32t_e}e^{-16t_e^2/e}.
\]
This proves \eqref{eq:one-tail}.  Exact symmetry of $X_e$ gives equal upper
and lower tails, which are disjoint for positive thresholds, and therefore
proves \eqref{eq:two-tail}.
\end{proof}

\subsection{A full moderate-deviation principle}

\begin{theorem}[Moderate deviations]\label{thm:mdp}
For every $a_e\to\infty$ with $a_e=o(\sqrt e)$, the variables
\[
 Y_e=\frac{X_e}{a_e\sqrt{e/32}}
\]
satisfy an LDP with speed $a_e^2$ and good rate $J(y)=y^2/2$.
\end{theorem}

\begin{proof}
The lattice of $Y_e$ is
\[
 L_e=\Delta_e(\mathbb Z+\theta_e),\qquad
 \Delta_e=\frac{\sqrt{32}}{a_e\sqrt e},
\]
for some shift $\theta_e\in[0,1)$.  Since $a_e=o(\sqrt e)$,
\begin{equation}
  a_e^2\Delta_e=\sqrt{32}\,\frac{a_e}{\sqrt e}\longrightarrow0.
  \label{eq:mdp-mesh-small}
\end{equation}
For $y$ in a fixed bounded interval, the corresponding centred displacement is
$x=a_e\sqrt{e/32}\,y+O(1)$.  Since $a_e=o(\sqrt e)$, one has $x=o(e)$
uniformly on that interval, and Lemma~\ref{lem:interior-support} shows that
all of its lattice points are attainable for all sufficiently large $e$.
Lemma~\ref{lem:rate} then gives uniformly
\begin{equation}
 eI\!\left(\frac12+\frac{x}{e}\right)
 =\frac12a_e^2y^2+o(a_e^2),
 \qquad
 G\!\left(\frac12+\frac{x}{e}\right)=8(1+o(1)).
 \label{eq:mdp-rate}
\end{equation}
The exact identity
\[
 \frac4{\sqrt{\pi e}}=\frac{a_e\Delta_e}{\sqrt{2\pi}}
\]
therefore turns the point probabilities into a Gaussian mesh measure:
uniformly for bounded $y$,
\begin{equation}
 \mathbb P(Y_e=y)
 =e^{o(a_e^2)}
 \frac{a_e}{\sqrt{2\pi}}e^{-a_e^2y^2/2}\,\Delta_e.
 \label{eq:mdp-mesh-mass}
\end{equation}
This display makes explicit why no condition such as
$\log e=o(a_e^2)$ is needed: the $e^{-1/2}$ point prefactor is paired with the
$O(a_e\sqrt e)$ lattice points in a fixed $Y$-interval.

We record the elementary shifted-mesh Laplace step.  If
$L_n=\Delta_n(\mathbb Z+\theta_n)$, $a_n\to\infty$,
$\Delta_n\downarrow0$ and $a_n^2\Delta_n\to0$, then for
$g_n(y)=a_n(2\pi)^{-1/2}e^{-a_n^2y^2/2}$,
\begin{align*}
 \limsup_{n\to\infty}\frac1{a_n^2}
 \log\sum_{y\in L_n\cap F}g_n(y)\Delta_n
 &\le-\inf_{y\in F}\frac{y^2}{2},\\
 \liminf_{n\to\infty}\frac1{a_n^2}
 \log\sum_{y\in L_n\cap O}g_n(y)\Delta_n
 &\ge-\inf_{y\in O}\frac{y^2}{2}
\end{align*}
for every bounded closed $F$ and bounded open $O$, uniformly in the shifts.
For the upper bound, attach one mesh cell to each selected point; by
$a_n^2\Delta_n\to0$ the logarithmic variation of $g_n$ across a bounded cell
is $o(a_n^2)$.  For the lower bound, choose a compact interval
$J\subset O$ whose rate infimum is arbitrarily close to that of $O$ and sum
all full mesh cells contained in $J$.  At most two endpoint cells are lost.
Thus the lower bound uses interval mass rather than a single lattice atom.
Applying this lemma to \eqref{eq:mdp-mesh-mass} gives the LDP lower bound for
bounded open sets and upper bound for bounded closed sets.

It remains to prove exponential tightness.  Fix
$\delta\in(0,1/6)$.  The central point bound
\eqref{eq:count-central-bound}, after division by \eqref{eq:Ue}, becomes
\[
 \mathbb P(Y_e=y)
 \le C a_e\Delta_e\,e^{-c_0a_e^2y^2}
\]
while $|X_e|\le\delta e$.  Summing by the same mesh-cell comparison gives
\[
 \limsup_{e\to\infty}\frac1{a_e^2}
 \log\mathbb P\!\left(M<|Y_e|\le
 \frac{\delta\sqrt{32e}}{a_e}\right)
 \le-c_1M^2.
\]
Outside the central density interval, \eqref{eq:count-far-bound} and
\eqref{eq:Ue} give a fixed exponential factor $e^{-\kappa e}$ up to a
polynomial.  Since $a_e^2=o(e)$, that contribution is superexponentially small
at speed $a_e^2$.  Exponential tightness follows, upgrading the bounded-set
bounds to the full LDP with good rate $J(y)=y^2/2$.
\end{proof}

\subsection{\texorpdfstring{Speed-$e$}{Speed-e} large deviations, the CLT, and moments}

\begin{theorem}[Global limits]\label{thm:global}
The variables $A_e=(V_e-1)/e$ satisfy an LDP with speed $e$ and good rate
$I$.  Furthermore
\[
 Z_e:=\frac{X_e}{\sqrt{e/32}}\Rightarrow N(0,1),
\]
and for every fixed integer $r\ge1$,
\[
 \mathbb E Z_e^{2r}\longrightarrow (2r-1)!!,
 \qquad
 \mathbb E X_e^{2r+1}=0.
\]
In particular $\operatorname{Var}(V_e)\sim e/32$.
\end{theorem}

\begin{proof}
We begin with the speed-$e$ LDP.  The full-support bound used in the proof of
Theorem~\ref{thm:counts}, together with \eqref{eq:Ue}, gives for every feasible
$i+j=e$ and all sufficiently large $e$,
\begin{equation}
 \mathbb P(A_e=i/e)
 \le C_0 e^{19/2}e^{-eI(i/e)}.
 \label{eq:ldp-point-upper}
\end{equation}
There are at most $e+1$ support points.  Consequently, for a closed set $F$,
\[
 \mathbb P(A_e\in F)
 \le C_0 e^{21/2}
 \exp\{-e\inf_{\alpha\in F\cap\operatorname{supp}(A_e)}I(\alpha)\}.
\]
The polynomial prefactor disappears after division of the logarithm by $e$.
Compactness of the effective domain and continuity of $I$ yield the closed-set
upper bound.

For the lower bound let $O$ be open and suppose $\inf_OI<\infty$.
Given $\varepsilon>0$, choose
$\alpha\in O\cap(1/3,2/3)$ with
$I(\alpha)\le\inf_OI+\varepsilon$; if the infimum occurs at an endpoint, take
$\alpha$ from the interior sufficiently close to it.  Let $i_e$ be a nearest
integer to $\alpha e$ and $j_e=e-i_e$.  Since $i_e/e\to\alpha$ and
$\alpha\in(1/3,2/3)$, Lemma~\ref{lem:interior-support} gives
$\widetilde u_{i_e,j_e}>0$ for all sufficiently large $e$.  Thus these are
actual feasible points, and Theorem~\ref{thm:bulk} gives
\[
 \lim_{e\to\infty}\frac1e
 \log\mathbb P(A_e=i_e/e)=-I(\alpha).
\]
Letting $\varepsilon\downarrow0$ proves the open-set lower bound.  Thus $I$ is
a good rate function on $[1/3,2/3]$, extended by $+\infty$ outside.

We next prove the CLT and moment convergence.  Set
$Z_e=X_e/\sqrt{e/32}$ and $\Delta_e=\sqrt{32/e}$.  A fixed bounded
$z$-interval corresponds to $|X_e|=O(\sqrt e)=o(e)$, so
Lemma~\ref{lem:interior-support} makes every point of the shifted lattice in
that interval attainable for all sufficiently large $e$.  The local theorem
therefore gives uniformly
\[
 \mathbb P(Z_e=z)=\varphi(z)\Delta_e(1+o(1)).
\]
Riemann sums therefore converge on compact intervals.  For uniform tail
control, \eqref{eq:count-central-bound} divided by \eqref{eq:Ue} gives, for
$|X_e|\le\delta e$,
\begin{equation}
 \mathbb P(Z_e=z)\le C\Delta_e e^{-c z^2},
 \label{eq:gaussian-domination-Z}
\end{equation}
with constants independent of the parity shift.  The remaining support is
bounded by a fixed polynomial times $e^{-\kappa e}$.  The domination
\eqref{eq:gaussian-domination-Z} proves tightness before passing to the limit,
and the compact Riemann sums then give $Z_e\Rightarrow N(0,1)$.

For any fixed $r\ge1$, multiplying \eqref{eq:gaussian-domination-Z} by
$|z|^{2r}$ and using monotone Gaussian-tail comparison yields a bound
independent of $e$ whose tail tends to zero as the truncation level tends to
infinity.  The exponentially small outer density region contributes
$o(1)$ even after multiplication by the maximal polynomial size of
$|Z_e|^{2r}$.  Hence $|Z_e|^{2r}$ is uniformly integrable, and compact
Riemann-sum convergence gives
\[
  \mathbb E Z_e^{2r}\longrightarrow(2r-1)!!.
\]
Odd centred moments vanish exactly by Proposition~\ref{prop:duality}.  Taking
$r=1$ yields $\operatorname{Var}(V_e)\sim e/32$.  Finally,
\[
 \sqrt e\left(\frac{V_e}{e+2}-\frac12\right)
 =\frac{e}{e+2}\frac{Z_e}{\sqrt{32}}
 \Rightarrow N(0,1/32),
\]
which is \eqref{eq:ruedinger-answer}.
\end{proof}

\section{Endpoint geometry and sharp probabilities}\label{sec:endpoints}

The extreme points of the support have a direct combinatorial meaning.

\begin{proposition}[Support at multiples of three]\label{prop:end-support}
If $e=3m$, $m\ge2$, then
\[
  m+2\le V_{3m}\le 2m.
\]
The lower endpoint is attained by a simplicial (triangulated) polytope; its
dual is cubic and attains the upper endpoint.  For $m=2$ the two endpoints
coincide at the tetrahedron.
\end{proposition}

\begin{proof}
The inequalities follow from \eqref{eq:support-ineq}.  Starting with a
tetrahedron, stack a new tetrahedron on a triangular face.  Each stacking step
adds one vertex and three edges and replaces one triangular face by three, so
it preserves simpliciality.  After $m-2$ steps the resulting polytope has
$e=3m$ and $V=m+2$.  Its polar dual has the same number of edges and
$V=F=3m+2-(m+2)=2m$, and every vertex of the dual has degree three.
\end{proof}

\begin{theorem}[Endpoint probability]\label{thm:endpoint}
For $m\to\infty$,
\[
 \mathbb P(V_{3m}=m+2)
 =\mathbb P(V_{3m}=2m)
 \sim \frac{6561}{32\sqrt2}\left(\frac4{27}\right)^m.
\]
\end{theorem}

\begin{proof}
At the lower endpoint,
\[
  i=m+1,\qquad j=2m-1.
\]
The two shifted binomial factors in \eqref{eq:bw-unrooted} simplify to
\[
 \binom{2i}{j+3}=\binom{2m+2}{2m+2}=1,
 \qquad
 \binom{2j}{i+3}=\binom{4m-2}{m+4}.
\]
Stirling's formula, or comparison with $\binom{4m}{m}$, gives
\[
 \binom{4m-2}{m+4}
 \sim \frac{729}{16}\sqrt{\frac{2}{3\pi m}}
       \left(\frac{256}{27}\right)^m.
\]
Substitution into \eqref{eq:bw-unrooted} therefore yields
\begin{equation}
 \widetilde u_{m+1,2m-1}
 \sim \frac{\sqrt6}{384\sqrt\pi}\,
 m^{-7/2}\left(\frac{256}{27}\right)^m.
 \label{eq:end-count}
\end{equation}
On the other hand, \eqref{eq:Ue} at $e=3m$ gives
\[
 U_{3m}\sim
 \frac{64^m}{486\sqrt\pi\,(3m)^{7/2}}.
\]
Dividing \eqref{eq:end-count} by this normalisation gives
\[
 \frac{\sqrt6}{384\sqrt\pi}\,
 486\sqrt\pi\,3^{7/2}
 \left(\frac{256/27}{64}\right)^m
 =\frac{6561}{32\sqrt2}\left(\frac4{27}\right)^m.
\]
Exact equality of the lower and upper endpoint probabilities follows from
polar duality, not merely asymptotically.
\end{proof}

\section{Discussion, provenance and reproducibility}\label{sec:discussion}

The fixed-edge law is a one-dimensional slice through the classical
vertex--face enumeration.  The slice passes through several asymptotic regimes.
In the bulk, the two shifted binomial coefficients generate the strictly convex
entropy function $I$.  Near the central saddle, the curvature
$I''(1/2)=32$ fixes the variance constant $1/32$, while successive Taylor
terms identify the $e^{3/4}$ Gaussian-relative window and the $e^{5/6}$
quartic window.  At the support boundary the interior prefactor vanishes and
one shifted binomial coefficient collapses to $1$, which is why the endpoint
probability requires a direct calculation rather than extrapolation of the
bulk formula.

Equation~\eqref{eq:ruedinger-answer} resolves the fixed-edge rescaled
limit-distribution question posed by R\"udinger
\cite{Ruedinger2009,Ruedinger2010}.  The local limits, exact moderate tails,
MDP, LDP, moment convergence and endpoint probabilities give a more detailed
description of the same fixed-edge distribution.  Related limit-law results
for other random planar graph and map parameters include
\cite{GimenezNoy2009,ColletDrmotaKlausner2019}.

The rooted two-parameter formula is due to Bender and Wormald
\cite{BenderWormald1988}, and their published unrooted asymptotic uses the
earlier uniform asymmetry/unrooting theorem \cite{BenderWormald1985}.
Section~\ref{sec:transfer} derives the fixed-edge unrooted formula from these
two inputs before the probabilistic analysis.  The exponentially accurate
rooted--unrooted comparison then transfers the resulting probability
statements between the two models.  Earlier fixed-edge total enumeration goes
back to Richmond and Wormald \cite{RichmondWormald1982}.

Within the literature searched, we found no previous explicit statement
combining the fixed-edge CLT with the local-limit, deviation, precise-tail and
endpoint results proved here.  We make no broader claim of priority.

Direct differentiation of \eqref{eq:rate-I} verifies \eqref{eq:Isecond}; expansion
about $1/2$ gives the coefficients $16$, $320/3$ and $23296/15$ in
\eqref{eq:I-taylor}; and $G(1/2)=8$ yields the central point prefactor
$4/\sqrt{\pi e}$.  The normalising constant in \eqref{eq:Ue} follows
independently from the Gaussian lattice sum.  At the endpoint, simplifying the
shifted binomial coefficients before applying Stirling gives the constant in
\eqref{eq:endpoint-prob}.  These computations check the displayed analytic formulae and are not used as
independent inputs to any proof.

\section*{Declarations}

\paragraph{Funding}
This research did not receive any specific grant from funding agencies in the
public, commercial, or not-for-profit sectors.

\paragraph{Declaration of competing interest}
The author declares that he has no known competing financial interests or
personal relationships that could have appeared to influence the work reported
in this paper.

\paragraph{Data availability}
No datasets were generated or analysed for the research described in this
article.  The results are analytic consequences of the published enumeration
formulae cited above.

\paragraph{Declaration of generative AI and AI-assisted technologies in the manuscript preparation process}
During the preparation of this work, the author used ChatGPT (OpenAI) in order
to assist with proof checking, literature organisation, and language editing.
After using this tool, the author reviewed and edited the content as needed,
independently verified the mathematical arguments, calculations, and citations,
and takes full responsibility for the content of the published article.

\end{document}